\documentclass[12pt, english, a4paper]{amsart}

\usepackage{graphicx}
\usepackage{amsmath,amssymb}
\usepackage[left=2.5cm, top=3cm, bottom=3cm, right=2.5cm]{geometry}

\usepackage{url,hyperref}

\usepackage{comment}
\theoremstyle{plain}
\newtheorem{theorem}{Theorem}[section]

\newtheorem{lemma}[theorem]{Lemma}
\newtheorem{corollary}[theorem]{Corollary}

\theoremstyle{definition}

\newtheorem{remark}[theorem]{Remark}

\newtheorem*{theorem*}{Theorem}

\newcommand{\CC}{{\mathbb C}}

\newcommand{\HH}{{\mathbb H}}

\newcommand{\PP}{{\mathbb P}}

\newcommand{\FS}{\mathbf{F}}

\def\Um{\Lambda(M)}

\hypersetup{colorlinks=true, urlcolor=blue}

\usepackage[hyperpageref]{backref}

\title
[K\"ahler submanifolds and the Umehara algebra]
{K\"ahler submanifolds and the Umehara algebra}

\author{Xiaoliang Cheng, Antonio J. Di Scala and Yuan Yuan}
\address{\normalfont Xiaoliang Cheng, Department of Mathematics, Jilin Normal University, Siping 136000, China.}
\email{chengxiaoliang92@163.com}

\address{\normalfont Antonio J. Di Scala, Dipartimento di Scienze Matematiche `G.L. Lagrange', Politecnico di Torino, Corso Duca degli Abruzzi 24, 10129, Torino, Italy.}
\email{antonio.discala@polito.it}

\address{\normalfont Yuan Yuan, Department of Mathematics, Syracuse University, Syracuse, NY 13244, USA.}
\email{yyuan05@syr.edu}

\thanks{X. Cheng is supported in part by National Natural Science Foundation of China grant NSFC-11301215; A.J. Di Scala is member of GNSAGA of INdAM, Italia; Y. Yuan is supported in
part by National Science Foundation grant DMS-1412384 and the seed grant program at Syracuse University.}
\date{}

\keywords{complex submanifolds, Fubini-Study spaces, indefinite K\"ahler metrics, isometric embedding, Nash algebraic}
\subjclass[2010]{32H02, 32Q40, 53B35}

\begin{document}

\maketitle

\begin{abstract}
We show that the indefinite complex space form $\mathbb{C}^{r, s}$ is not a relative to the
indefinite complex space form $\CC\PP_l^N(b)$ or $ \CC\HH_m^N(b)$.
We further study whether two Fubini-Study spaces are relatives or not.
\end{abstract}

\section{Introduction}

Problems about holomorphic and isometric embeddings are classical questions in complex and differential geometry.
Starting with Bochner's paper \cite{[B]} such questions have been studied extensively by
many authors e.g. \cite{[C], [CU], [DL2], [HJL], [HY1], [Mo3], [Ng], [U1], [YZ]}. In his PhD. Thesis \cite{[C]}, E. Calabi obtained
the existence, uniqueness and global extension of a local holomorphic
isometry from a complex K\"ahler manifold into a complex space form (also called Fubini-Study spaces), among many other important
results. Calabi's results show that the complex version of Nash's theorem is not true as it was recently asked in \cite[mathoverflow]{[mathoverflow]}.
In particular Calabi proved that any Fubini-Study space cannot
be locally isometrically embedded into another Fubini-Study space
with a different curvature sign with respect to the canonical
K\"ahler metrics, and he further gave the sufficient and necessary condition for one Fubini-Study space to be embedded into another one.\\

The key object in Calabi's work is his \emph{diastasis function}. Unlike the K\"ahler potential, Calabi's diastasis is unique.
Actually, the diastasis is a clever choice of a potential around each point of an analytic K\"ahler manifold.
Thanks to the diastasis Calabi was able to reduced metric tensor equations to
functional identities. This idea turns out to be quite useful in problems with holomorphic and isometry immersions, quantization problems e.g. \cite{[DIL]}, \cite[page 63]{[D&Q]} and even in related questions of number theory \cite{[CU], [Mo2], [Mo3]}.\\

Umehara \cite{[U2]} later generalized Calabi's existence and uniqueness results for holomorphic isometries from a complex manifold with an indefinite K\"ahler metric into an indefinite complex space form.
On the other hand, Umehara \cite{[U1]} studied an interesting question
whether two complex space forms can share a common submanifold with
the induced metrics. Following Calabi's idea, Umehara proved that
two complex space forms with different curvature signs cannot share
a common K\"ahler submanifold \cite{[U1]}\cite{[U2]}.\\

Two K\"ahler manifolds are called \emph{relatives} when they share a
common K\"ahler manifold i.e. a complex submanifold of one of them endowed with the induced metric is biholomorphically isometric to a complex submanifold of the other endowed with the induced metric.
It was shown in \cite{[DL2]} that Hermitian symmetric spaces of different compact types are not relatives.
In addition, the fact that Euclidean spaces and Hermitian symmetric spaces of compact types are not relatives follows from Umehara's result \cite{[U1]} and the classical
Nakagawa-Takagi embedding of Hermitian symmetric spaces of compact type into complex projective spaces.
Finally, it was shown in \cite{[HY2]} that Euclidean spaces and Hermitian symmetric spaces
of non-compact types are not relatives.\\

In this paper, we consider the relativity problems for two indefinite complex space forms as well as for two Fubini-Study spaces.
We show that the flat indefinite complex space form $\CC^{N,s}$ is not a relative with the non flat indefinite complex space forms  $\CC\PP_{s'}^{N'}(b)$, $ \CC\HH_{s'}^{N'}(b)$ (see Corollary \ref{corollary} below). The relativity problem can be reformulated in terms of the so called Umehara's algebra.
We use the techniques developed in \cite{[HY1]} and \cite{[HY2]} to obtain non-trivial improvements of Umehara's results \cite{[U2]} (see Theorem \ref{UA} below).
In the last section we give necessary and sufficient conditions for two Fubini-Study space forms $\FS(n,b)$ and $\FS(m,a)$ to be relatives.

\bigskip

{\bf Acknowledgement}: We thank Ming Xiao for pointing out a mistake in the previous version of our paper. This material is based upon work supported by the National Science Foundation under Grant No. 0932078 000 while the third author was in residence at the Mathematical Sciences Research Institute in Berkeley, California, during the 2016 Spring semester.


\section{The Umehara algebra and the relatively problem}

Umehara introduces in \cite{[U2]}  the associate algebra $\Um$ of a complex manifold $M$. Since the interest here is the local existence of a K\"ahler submanifold at some point $p$ in $M$. We modify Umehara's definition as follows. Let $\mathcal{O}_p$ denote the local ring of germs of holomorphic functions at $p$.  
Define 
\[ \wedge_p := \left\{ f : f = \sum_{i=1}^{n} r_i|\chi_i|^2 , \chi_i \in \mathcal{O}_p , r_i \in \mathbb{R} \right\} , \]
and
let $K_p$ be the field of fractions of $\wedge_p$. Notice that the germs of real numbers, denoted by $\mathbb{R}_p$, belong to $K_p$.
%
%
The main result is following local characterization of Umehara's algebra.

\begin{theorem}\label{UA}
Let $p$ be a fixed point on a complex manifold $M$ and let $\chi_1, \cdots, \chi_s \in \mathcal{O}_p$ be non-constant germs of holomorphic functions at $p$ such that $\chi_1(p)= \cdots = \chi_s(p)=0$. For non-zero real numbers $r_1, \cdots, r_s$, the following statements hold:

\begin{itemize}

\item[(i)] $\log \left(1+ \sum_{i=1}^s r_i |\chi_i|^2\right) \not\in K_p \setminus \mathbb{R}_p$;

\item[(ii)] $\exp \left( \sum_{i=1}^s r_i |\chi_i|^2\right) \not\in K_p \setminus \mathbb{R}_p$;

\item[(iii)] \rm{If} $\left(1+ \sum_{i=1}^s r_i |\chi_i|^2\right)^{\alpha} \in K_p \setminus \mathbb{R}_p$, then $\alpha \in \mathbb{Q}$.

\end{itemize}
\end{theorem}

The indefinite complex Euclidean space $\CC^{N,s}(0 \leq s \leq N)$ is the complex linear space $\CC^N$ with the indefinite K\"ahler metric $$\omega_{\CC^{N,s}} = \sqrt{-1} \left( \sum_{i=1}^{N-s} d z_i \wedge d\bar z_i - \sum_{j=1}^s d z_{N-s +j} \wedge d \bar z_{N-s+j} \right) = \sqrt{-1} \partial \bar\partial \left(\sum_{i=1}^{N-s} |z_i|^2 - \sum_{j=1}^s |z_{N-s+j}|^2 \right).$$
%
%
The indefinite complex projective space $\CC\PP_s^N(b)(0\leq s \leq N)$ of constant holomorphic sectional curvature $b>0$ is the open submanifold $\{(\xi_0, \cdots, \xi_N) \in \mathbb{C}^{N+1}: \sum_{i=0}^{N-s} |\xi_i|^2 - \sum_{j=0}^{s-1} |\xi_{N-j}|^2 \} / \CC^*$ of $\CC\PP^N$. The indefinite K\"ahler metric of $\CC\PP_s^N(b)$ is given by $$\omega_{\CC\PP_s^N(b)} = \frac{\sqrt{-1}}{b} \partial \bar\partial \log \left( \sum_{i=0}^{N-s} |\xi_i|^2 - \sum_{j=0}^{s-1} |\xi_{N-j}|^2\right) .$$ The indefinite complex hyperbolic space
 $ \CC\HH_s^N(b) (0 \leq s \leq N)$ of constant holomorphic sectional curvature $b<0$ is obtained from $\CC\PP^N_{N-s}(-b)$ with indefinite K\"ahler metric
 $$ \omega_{\CC\HH_s^N(b)} = - \frac{\sqrt{-1}}{(-b)} \partial \bar\partial \log \left( \sum_{i=0}^{s} |\xi_i|^2 - \sum_{j=0}^{N- s-1} |\xi_{N-j}|^2\right).$$ 
Under inhomogeneous coordinates $(z_1, \cdots, z_N) = (\xi_1 / \xi_0, \cdots, \xi_N / \xi_0)$, for instance, assuming  $\xi_0 \not=0 $, the metrics are given by
$$\omega_{\CC\PP_s^N(b)} = \frac{\sqrt{-1}}{b} \partial \bar\partial \log \left( 1+ \sum_{i=1}^{N-s} |z_i|^2 - \sum_{j=0}^{s-1} |z_{N-j}|^2\right) $$ and
$$\omega_{\CC\HH_s^N(b)} = - \frac{\sqrt{-1}}{(-b)} \partial \bar\partial \log \left(1+ \sum_{i=1}^{s} |z_i|^2 - \sum_{j=0}^{N- s-1} |z_{N-j}|^2\right).$$
In particular, when $s=0$, $\CC^{N,0}, \CC\PP_s^N(b), \CC\HH_s^N(b)$ are just the standard complex Euclidean, projective, hyperbolic space, respectively.

\medskip

Furthermore, suppose that $D$ is a complex manifold such that  there exist  holomorphic maps $F=(F_1, \cdots, F_N): D \rightarrow \CC^{N,s}$ and $L=(L_1, \cdots, L_{N'}): D \rightarrow \CC\PP_{s'}^{N'}(b)$ (or $\CC\HH_{s'}^{N'}(b)$) with $F^*\omega_{\CC^{N,s}} = L^*\omega_{\CC\PP_{s'}^{N'}(b)}$ (or $F^*\omega_{\CC^{N,s}} = L^* \omega_{\CC\HH_{s'}^{N'}(b)}$) on $D$.
Fixing $x \in D$, without loss of generality, we assume $F(x)=0, L(x)=0$ by composing the automorphisms on $ \CC^{N,s}$ and $\CC\PP_{s'}^{N'}(b)$ (or $\CC\HH_{s'}^{N'}(b)$). By the standard argument to get rid of $\partial\bar\partial$ as in the proof of Lemma \ref{al}, we have
$$ \log \left( 1+ \sum_{i=1}^{N'-s'} |L_i(z)|^2 - \sum_{j=0}^{s'-1} |L_{N'-j}(z)|^2\right) = b \left( \sum_{i=1}^{N-s} |F_i(z)|^2 - \sum_{j=1}^s |F_{N-s+j}(z)|^2 \right)$$ or
$$ \log \left(1+ \sum_{i=1}^{s'} |L_i(z)|^2 - \sum_{j=0}^{N'- s'-1} |L_{N'-j}(z)|^2\right) =b \left( \sum_{i=1}^{N-s} |F_i(z)|^2 - \sum_{j=1}^s |F_{N-s+j}(z)|^2 \right).$$
This contradicts to Theorem \ref{UA} (i).
Therefore, we proved the following corollary.

\begin{corollary}\label{corollary}
Let $(D, \omega_D)$ be a K\"ahler submanifold of  $\CC^{N,s}$. Then any open subset of $D$ cannot be a K\"ahler submanifold of $\CC\PP_{s'}^{N'}(b)$ or that of $ \CC\HH_{s'}^{N'}(b)$. In other words, $\CC^{N,s}$ and $\CC\PP_{s'}^{N'}(b)$ (or $ \CC\HH_{s'}^{N'}(b)$) cannot be relatives.
\end{corollary}


\section{Indefinite complex space forms}

\begin{lemma}[cf. Theorem 3.2 in  \cite{[U2]}]\label{al}
Let $h_\sigma, k_\tau$ be the germs of holomorphic functions at $p \in U$ with $h_\sigma(p)=k_\tau(p)=0$ for $1\leq \sigma \leq r, 1\leq \tau \leq s$. Then there exist linearly independent germs of holomorphic functions $h'_\sigma, k'_\tau$ for $1\leq \sigma \leq r', 1\leq \tau \leq s'$ such that
$$\sum_{\sigma=1}^r | h_\sigma |^2-\sum_{\tau=1}^s| k_{\tau} |^2 = \sum_{\sigma=1}^{r'} | h'_\sigma |^2-\sum_{\tau=1}^{s'} | k'_{\tau} |^2.$$
\end{lemma}
\begin{proof}
Choose a  holomorphic coordinate $\{z\}$ at $p \in U $ such that $z(p)=0$ and define $F(z)= \sum_{\sigma=1}^r| h_\sigma(z) |^2-\sum_{\tau=1}^s| k_\tau (z)|^2 $. It follows from the definition that $F(z)$ is a real analytic function of finite rank. By Theorem 3.2 in \cite{[U2]}, there exit a pair of non-negative numbers $(r', s')$, a germ of holomorphic function $\phi_0$ and germs of linearly independent holomorphic functions $\phi_1,\cdots, \phi_N$ at $p$, such that
\begin{equation}\label{lemma}
F(z)=\text{Re}(\phi_0(z)) + \sum_{\sigma=1}^{r'}| \phi_\sigma(z) |^2-\sum_{\tau=1}^{s'}| \phi_{\tau+r'}(z) |^2,
\end{equation}
with  $\phi_1(0)=\cdots=\phi_{r'+s'}(0)=0.$ By comparing the Taylor expansion on the left and right sides of the equation (\ref{lemma}), it follows that $\phi_0$ is a constant function and thus $\phi_0 \equiv 0$.
\end{proof}

From now on, we can assume, without loss of generality, that $\{f_1, \cdots, f_l, g_1, \cdots, g_m\}$ and $\{h_1, \cdots, h_r, k_1, \cdots, k_s\}$ are sets of linearly independent holomorphic functions.
\medskip

{\bf Proof of Theorem \ref{UA}}
The idea of proof originates from \cite{[Hu1]} and \cite{[HY2]}. We prove the Theorem \ref{UA} (i) (ii) by contradiction. Choose a  holomorphic coordinate $\{z\}$ at $p \in U $ such that $z(p)=0$. For Part (i), suppose, on the contrary, that $\log \left(1+ \sum_{i=1}^n r_i |\chi_i|^2\right) \in K_p \setminus \mathbb{R}_p$. By rewriting $1+ \sum_{i=1}^n r_i |\chi_i(z)|^2 =1+ \sum_{\sigma=1}^r| h_\sigma(z) |^2-\sum_{\tau=1}^s| k_\tau (z)|^2$ with $\{h_1, \cdots, h_r, k_1, \cdots, k_s\}$ linearly independent, we may assume
\begin{equation}\label{albert1}
\log \left( 1+ \sum_{\sigma=1}^r| h_\sigma(z) |^2-\sum_{\tau=1}^s| k_\tau (z)|^2 \right) = \frac{\sum_{i=1}^l | f_i(z) |^2-\sum_{j=1}^m | g_j(z) |^2}{\sum_{i=1}^{l'} | f'_i(z) |^2-\sum_{j=1}^{m'} | g'_j(z) |^2}.\end{equation}
By intersecting $p$ with a certain one dimensional complex plane, we may assume that (\ref{albert1}) holds in an open set $U \subset \mathbb{C}$.
By polarization, (\ref{albert1}) is equivalent to
\begin{equation}\label{polar}
\log\left(1+\sum_{\sigma=1}^r h_\sigma(z) \bar h_\sigma(w) -\sum_{\tau=1}^s k_\tau (z) \bar k_\tau(w) \right)= \frac{\sum_{i=1}^l  f_i(z) \bar f_i(w)-\sum_{j=1}^m g_j(z) \bar g_j(w)}{\sum_{i=1}^{l'} f'_i(z) \bar f'_i(w)-\sum_{j=1}^{m'} g'_j(z) \bar g'_j(w)},
\end{equation}
where $(z, w) \in U\times\hbox{conj}({U})$, $\hbox{conj}({U})=\{z \in \CC | \bar z \in U\}$, and $\bar{\chi}_i(w) = \overline{\chi_i(\overline{w})}$.
\\  \
Taking $k$-th derivative of  the equation (\ref{polar}) in $w$ for $k=1,2,\cdots$, and then evaluating at $w=0$, we have the following matrix equation:
\begin{equation}\notag
P = A \cdot X + {\rm higher~order ~terms ~in~} X,
\end{equation}
where
\begin{equation}\notag
A=\begin{bmatrix}
 \cdots \frac{\partial \bar {h}_{\sigma}}{\partial w}(0) \cdots -\frac{\partial \bar {k}_{\tau}}{\partial w}(0) \cdots \\
\vdots\\
 \cdots \frac{\partial^k \bar {h}_{\sigma}}{\partial w^k}(0) \cdots -\frac{\partial^k \bar {k}_{\tau}}{\partial w^k}(0) \cdots \\
\vdots\\
\end{bmatrix}_{\infty\times(r+s),} ~
X=\begin{bmatrix}
\vdots\\ h_\sigma(z) \\ \vdots\\  k_\tau(z) \\ \vdots
\end{bmatrix}_{(r+s) \times 1,}
\text{and}~
P^t=\begin{bmatrix}
 \vdots\\ p_k \\  \vdots
\end{bmatrix}_{\infty\times1,}\end{equation} with each $p_k$ being rational function in $f_1(z),\cdots,f_l(z),g_1(z),\cdots,g_m(z)$ and \\
 $f'_1(z),\cdots,f'_{l'}(z),g'_1(z),\cdots,g'_{m'}(z)$.

We claim rank$(A)=r+s$, i.e. there exist $k=k_1, \cdots, k_{r+s}$ such that $k_1$-row to $k_{r+s}$-row in matrix $A$ are linearly independent and all other rows can be written as linear combinations of $k_1$-row up to $k_{r+s}$-row. Reorganize the matrices $A$ and $P$ by deleting rows other than $k_1$-row to $k_{r+s}$-row, denoted the corresponding matrices by $A_{r+s}, P_{r+s}$ respectively. We obtain the non-degenerate matrix equation $$P_{r+s} = A_{r+s} \cdot X + {\rm higher~order ~terms ~in~} X.$$
It follows by the implicit function theorem that each element in $X$ is a Nash algebraic function in $f_1(z),\cdots,f_l(z), g_1(z),\cdots,g_m(z)$ and  $f'_1(z),\cdots,f'_{l'}(z),g'_1(z),\cdots,g'_{m'}(z)$. Then one reach the contradiction by the similar argument in \cite{[HY2]} and the reader may refer to \cite{[HY2]} for the detailed proof. The idea is as follows. Suppose $f_1(z),\cdots,f_l(z),g_1(z),\cdots,g_m(z)$ and  $f'_1(z),\cdots,f'_{l'}(z),g'_1(z),\cdots,g'_{m'}(z)$ are all Nash algebraic functions in $z$. So are all $h_1, \cdots, h_r, k_1, \cdots, k_s$ by the above argument. Then the left hand side of the equation (\ref{albert1}) has logarithmic growth while the right hand side of the equation (\ref{albert1}) has polynomial growth as $z$ approaches the pole. If $f_1(z),\cdots,f_l(z),g_1(z),\cdots,g_m(z)$ and $f'_1(z),\cdots,f'_{l'}(z),g'_1(z),\cdots,g'_{m'}(z)$ are not all Nash algebraic functions. Then one can choose a maximal algebraic independent subset $S \subset \{f_1,\cdots,f_l,g_1,\cdots,g_m,  f'_1,\cdots,f'_{l'},g'_1,\cdots,g'_{m'}\}$ such that any element in $\{f_1,\cdots,f_l,g_1,\cdots,g_m, f'_1,\cdots,f'_{l'}, g'_1, \cdots, g'_{m'}, h_1, \cdots, h_r, k_1, \cdots, k_s\}$ is Nash algebraic functions in $z$ and elements in $S$.
Denote elements in $S$ by $\{X_1, \cdots, X_\kappa\}$ and we can write each $f_*(z), f'_*(z), g_*(z), g'_*(z), h_*(z), k_*(z)$ by Nash algebraic functions in $z, X$ given by $\hat f_*(z, X), f'_*(z, X), g_*(z, X), g'_*(z, X), h_*(z, X), k_*(z, X)$ respectively. By similar argument as in \cite{[HY2]}, we have
\begin{equation}\label{albert3}
\log \left(1+\sum_{\sigma=1}^r \hat h_\sigma(z, X) \bar h_\sigma(w) -\sum_{\tau=1}^s \hat k_\tau (z, X) \bar k_\tau(w) \right)= \frac{\sum_{i=1}^l  \hat f_i(z, X) \bar f_i(w)-\sum_{j=1}^m \hat g_j(z, X) \bar g_j(w)}{\sum_{i=1}^{l'} \bar f'_i(z, X) \bar f'_i(w)-\sum_{j=1}^{m'} \hat g'_j(z, X) \bar g'_j(w)}
\end{equation} for independent variables $z, w, X$.
Hence for fixed $w$ the left hand side of the equation (\ref{albert3}) has logarithmic growth while the right hand side of the equation (\ref{albert3}) has polynomial growth as $z$ approaches the pole.
We again reach a contradiction.

Now we show rank$(A)=r+s$.
Suppose rank$(A)=d<r+s$. Without loss of generality, we assume that the first $d $ columns are linearly independent in the
 coefficient matrix $A$,  writing $L_1,L_2,\cdots,L_d$. Then, for any $n$ with $d<n \leq r+s$, the $n$-th column is linear combination of $L_1,L_2,\cdots,L_d$, i.e.
\begin{equation}\notag\label{function1}
 L_n=\sum_{i=1}^{d}C_iL_i.
\end{equation}
In other words, the $n$-th element in $\{h_1, \cdots, h_r, k_1, \cdots, k_s\}$ can be written as linear combination of the first $d$ elements by the Taylor expansion, meaning $\{h_1, \cdots, h_r, k_1, \cdots, k_s\}$ is not linear independent. This is a contradiction. Thus we complete the proof of Theorem \ref{UA} (i).

Part (ii) follows from the similar argument. The only difference is to take logarithmic differentiation in $w$.

For Part (iii), assume \begin{equation}\label{equation11}
\left(1+\sum_{\sigma=1}^r h_\sigma(z) \bar h_\sigma(w) -\sum_{\tau=1}^s k_\tau (z) \bar k_\tau(w) \right)^\alpha= \frac{\sum_{i=1}^l  f_i(z) \bar f_i(w)-\sum_{j=1}^m g_j(z) \bar g_j(w)}{\sum_{i=1}^{l'} f'_i(z) \bar f'_i(w)-\sum_{j=1}^{m'} g'_j(z) \bar g'_j(w)}.
\end{equation}
By taking logarithmic differentiation in $w$ and applying the similar argument as for Part (i), we know that each element in $\{h_1(z), \cdots, h_r(z), k_1(z), \cdots, k_s(z)\}$
a Nash algebraic function in $f_1(z),\cdots,f_l(z), g_1(z),\cdots,g_m(z)$ and  $f'_1(z),\cdots,f'_{l'}(z),g'_1(z),\cdots,g'_{m'}(z)$. Suppose $f_1(z),\cdots,f_l(z),g_1(z),\cdots,g_m(z)$ and  $f'_1(z),\cdots,f'_{l'}(z),g'_1(z),\cdots,g'_{m'}(z)$ are all Nash algebraic functions in $z$. Then in (\ref{equation11}) we have a Nash algebraic function to the power of $\alpha$ is equal to another Nash algebraic function. Thus $\alpha \in \mathbb{Q}$. Otherwise, one can choose a maximal algebraic independent subset $S \subset \{f_1,\cdots,f_l,g_1,\cdots,g_m,  f'_1,\cdots,f'_{l'},g'_1,\cdots,g'_{m'}\}$ such that any element in $\{f_1,\cdots,f_l,g_1,\cdots,g_m, f'_1,\cdots,f'_{l'}, g'_1, \cdots, g'_{m'}, h_1, \cdots, h_r, k_1, \cdots, k_s\}$ is Nash algebraic functions in $z$ and elements in $S$.
By similar argument as above, we have
\begin{equation}\label{albert2}
\left(1+\sum_{\sigma=1}^r \hat h_\sigma(z, X) \bar h_\sigma(w) -\sum_{\tau=1}^s \hat k_\tau (z, X) \bar k_\tau(w) \right)^\alpha= \frac{\sum_{i=1}^l  \hat f_i(z, X) \bar f_i(w)-\sum_{j=1}^m \hat g_j(z, X) \bar g_j(w)}{\sum_{i=1}^{l'} \bar f'_i(z, X) \bar f'_i(w)-\sum_{j=1}^{m'} \hat g'_j(z, X) \bar g'_j(w)}
\end{equation} for independent variables $z, w, X$.
For fixed $w$, we have in (\ref{albert2}) a Nash algebraic function to the power of $\alpha$ is equal to another Nash algebraic function. Thus again $\alpha \in \mathbb{Q}$. This completes the proof of Theorem \ref{UA} (iii). \qed




\section{Fubini-Study spaces}

Along this section we use Calabi's original notation $\FS(n,b)$ for Fubini-Study spaces.
Namely, as in  \cite[pages 16 and 17]{[C]}, we denote with $\FS(n,b)$ a complex space form whose K\"ahler potential
is locally given by
\[ \frac{1}{b} \log(1 + b |Z|^2)\]
where $Z = (z_1, \cdots,z_n)$ and $|Z|^2 := \sum_{i=1}^n|z_i|^2$. So
\[  \FS(n,b) = \begin{cases} \mathbb{CP}^n_0(b) & \mbox{ if } b > 0 \, , \\
              \mathbb{CH}^n_0(b) & \mbox{ if } b < 0 \, .
\end{cases} \]

\begin{theorem}\label{relativesFS} Let $\FS(n,b)$ and $\FS(m,a)$ be two complex space forms
where $a,b \in \mathbb{R}^+$ are positive real numbers.
Assume that:
\begin{itemize}
\item[(i)] there are positive integers $s,r$ such that $s a = r b$,
\item[(ii)]  $m+n+1 >  \max \left\{ \binom{s+m}{s},\binom{r+n}{r}  \right\}$.
\end{itemize}
Then $\FS(n,b)$ and $\FS(m,a)$ are relatives.

\end{theorem}

For example $\FS(3,a)$ and $\FS(8,2a)$ are relatives for any $a > 0$.

\it Proof. \rm Let $\kappa = s a = rb > 0$ be the common value of the two numbers $sa$ and $rb$.
Then according to Calabi's \cite[Theorem 13, page 21]{[C]} both spaces $\FS(n,b)$ and $\FS(m,a)$ can be embedded into the bigger $\FS(N,\kappa)$ where
\[ N := \max \left\{ \binom{s+m}{s} -1 ,\binom{r+n}{r}-1 \right\} \, \, .\]
By Condition $(\mathrm{ii})$ and the well-known intersection theorem \cite[Theorem 7.2, page 48]{[Ha]} every irreducible component $Z$ of the intersection
$\FS(m,a) \bigcap \FS(n,b)$ has positive dimension. Hence $\FS(n,b)$ and $\FS(m,a)$ are relatives since $Z$ has an open subset of smooth points \cite[Theorem 5.3, page 33]{[Ha]}. $\Box$

\begin{remark}
Note that for positive integers $s, m$ with $s \leq m$, $\FS(m, s)$ and $\FS(m, 1)$ are relatives since $(\cdots, f_j(z), \cdots ): U \subset \mathbb{C} \rightarrow \mathbb{C}^m \subset \FS(m, s)$ and $(z, \cdots, z): U \subset \mathbb{C} \rightarrow \mathbb{C}^m \subset \FS(m, 1)$ with $f_j(z) = \sqrt{\frac{m^{j}}{s} \binom{s}{j}} z^j$ for $1 \leq j \leq s$ and $f_j(z) =0$ for $s < j \leq m$ satisfy  $$\frac{1}{s} \log \left( 1+ s \sum_j |f_j|^2 \right)=\log \left(1+ m |z|^2 \right).$$ Obviously, $2m+1 \ll \binom{m+s}{s}$ for $s \gg 1$. For instance, $\FS(2,2)$ and $\FS(2,1)$ are relatives but $2 + 2 + 1 = 5  < \binom{2 + 2}{2} = 6$. This example shows that (ii) in the above Theorem is not a necessary condition for $\FS(n,b)$ and $\FS(m,a)$ to be relatives. Actually, this can be explained as follows:  $\FS(1,1) \hookrightarrow \FS(n,1)$ by the totally geodesic embedding and $\FS(1,1) \hookrightarrow \FS(n,b)$ for $b \in \mathbb{N}$, $n \geq b \geq 1$ by Calabi's result \cite[Theorem 13, page 21]{[C]}. Thus, $\FS(n,b)$ and $\FS(n,1)$ are relatives but $2n + 1$ is not necessarily greater than $\binom{b + n}{b}$.
\end{remark}

\begin{theorem}\label{litchi}
Let $a \not=0, b\not=0$.
Suppose that $\FS(n,b)$ and $\FS(m,a)$ are relatives. Then $ab>0$ and $a/b \in \mathbb{Q}$.
\end{theorem}

\begin{proof}
Let $U \subset \mathbb{C}$ be a connected open set. Suppose that $\FS(n,b)$ and $\FS(m,a)$ are relatives. By composing with elements in holomorphic isometry groups, it is equivalent to the existence of holomorphic maps $H=(h_1, \cdots, h_m): U \rightarrow \FS(m, a), K=(k_1, \cdots, k_n): U \rightarrow \FS(n, a)$  with $H(0)=0, K(0)=0$ such that $$\frac{1}{a} \log \left(1+ a \sum_{i=1}^m |h_i(z)|^2 \right) = \frac{1}{b} \log \left(1+ b \sum_{i=1}^n |k_i(z)|^2 \right).$$ This is equivalent to
\begin{equation}\label{albert}
1+a \sum_{i=1}^m |h_i(z)|^2 = \left( 1+ b \sum_{i=1}^n |k_i(z)|^2 \right)^{a/b}.
\end{equation}
If $ab <0$, it follows from Umehara's argument \cite{[U2]} that the equation (\ref{albert}) cannot hold. Furthermore, $a/b \in \mathbb{Q}$ follows from Theorem \ref{UA}(iii).
\end{proof}

\begin{corollary}\label{intersection}
Let $a, b$ be two positive real number such that $\FS(n,b)$ and $\FS(m,a)$ are relatives.
Then there are $N \in \mathbb{N}$, $\kappa \in \mathbb{R}^+$ and holomorphic and isometric immersions $f: \FS(n,b) \to \FS(N,\kappa)$, $h: \FS(m,a) \to \FS(N,\kappa)$ such that \[ \mathrm{dim}(f(\FS(n,b)) \bigcap h(\FS(m,a))) > 0  \, \, .\]
\end{corollary}

\begin{proof} Since $\frac{a}{b} \in \mathbb{Q}$ there are $r,s \in \mathbb{N}$ such that $r a = s b$. Set $\kappa := r a = s b $.
Then by Calabi's Theorem \cite[Theorem 13, page 21]{[C]} there are holomorphic and isometric immersions $f: \FS(n,b) \to \FS(n',\kappa)$ and
$j: \FS(m,a) \to \FS(m',\kappa)$. Taking $N := \max(m',n')$ we can assume that $N=m'=n'$.
Since $\FS(n,b)$ and $\FS(m,a)$ are relatives there is a K\"ahler manifold $Z$, $\mathrm{dim}(Z) > 0$ and holomorphic and isometric immersions $G: Z \to \FS(n,b)$ and $H:Z \to \FS(m,a)$. Then $ f \circ G$ and $j \circ H$ are holomorphic and isometric immersions of $Z$ into $\FS(N,\kappa)$. By Calabi's rigidity \cite[Theorem 9, page 18]{[C]} there is an isometry $u \in \mathrm{U}(N+1)$ such that $f \circ G = u \circ j \circ H$. So by setting $h := u \circ j$ we get that $f(G(Z))= h(H(Z)) \subset f(\FS(n,b)) \bigcap h(\FS(m,a))$ and $\mathrm{dim}(f(G(Z))) = \mathrm{dim}(Z) > 0 $.
\end{proof}





\begin{corollary}\label{necessary}
Let $a, b$ be two positive real number such that $a/b \in \mathbb{Q}$. Suppose that $\FS(n,b)$ and $\FS(m,a)$ are relatives. Moreover, assume $sa=rb$ for positive integers $s, r$ with $(s, r)=1$. Then \begin{equation}\label{111}
r+1 \leq \binom{s+m}{s} ~\rm{and}~ s+1 \leq \binom{r+n}{r}.
\end{equation}
\end{corollary}

\begin{proof}
It follows from the proof of Theorem \ref{litchi} that $\FS(n, b)$ and $\FS(m,a)$ are relatives if and only if there exist holomorphic maps $H=(h_1, \cdots, h_m): U \subset \mathbb{C} \rightarrow F(m, a), K=(k_1, \cdots, k_n): U\subset \mathbb{C}  \rightarrow \FS(n, a)$  with $H(0)=0, K(0)=0$ such that
\begin{equation}\notag
\left(1+a \sum_{i=1}^m |h_i(z)|^2\right)^s = \left( 1+ b \sum_{i=1}^n |k_i(z)|^2 \right)^{r}.
\end{equation}
Note that the left hand side can be written as the sum of 1 and norm squares of $P$ linearly independent holomorphic functions with $s \leq P \leq \binom{m+s}{s}-1$ and the right hand side can be written as the sum of 1 and norm squares of $Q$ linearly independent holomorphic functions with $r \leq Q \leq \binom{n+r}{r}-1$. One knows
$P=Q$ by Calabi's theorem (cf. also \cite{[DA]}). Thus we conclude (\ref{111}).
\end{proof}

\subsection{Conditions (\ref{111}) in Corollary \ref{necessary} are not sufficient.}

Here we show that $\FS(2,1)$ and $\FS(2,\frac{3}{2})$ are not relatives. Observe that conditions (\ref{111}) in Corollary \ref{necessary} hold true.
We also give a necessary condition on $a,b$ for $\FS(2,a)$ and $\FS(2,b)$ to be relatives.\\

To show that that $\FS(2,1)$ and $\FS(2,\frac{3}{2})$ are not relatives it is enough to show that there are no $h_1,h_2,k_1,k_2$ in the local ring $\mathcal{O}_{\mathbb{C},0}$ parameterizing local curves $\{h_1, h_2\}$ and $\{k_1, k_2\}$ through $(0,0) \in \mathbb{C}^2$ such that
\begin{equation}\label{2-3} (1 + |k_1(z)|^2 + |k_2(z)|^2)^2 = (1 + |h_1(z)|^2 + |h_2(z)|^2)^3  \, \, \end{equation}
Expanding both sides we get
\[ |\sqrt{2} k_1 |^2 + |k_1^2|^2 + |\sqrt{2} k_2|^2 + |\sqrt{2} k_1 k_2|^2 + |k_2^2|^2 = \] \[ |\sqrt{3} h_1|^2 + |\sqrt{3} h_1^2|^2 + |h_1^3|^2 + |\sqrt{3} h_2|^2 + |\sqrt{6} h_1 h_2|^2 + |\sqrt{3} h_1^2 h_2|^2 + |\sqrt{3} h_2^2|^2 + |\sqrt{3} h_1 h_2^2|^2 + |h_2^3|^2 \]
where we omitted the letter $z$ from the argument of the functions.\\

According to \cite[Proposition 3, page 102]{[DA]} (or by \cite[Theorem 2, page 8]{[C]}) there is a matrix $U \in \mathbf{U}(9)$ such that
\begin{equation}\label{curves}  U \cdot H = K \end{equation}
 where \[ H := [\sqrt{3} h_1 , \sqrt{3} h_2 , \sqrt{3} h_1^2 , \sqrt{6} h_1 h_2, \sqrt{3} h_2^2, h_1^3, \sqrt{3} h_1^2 h_2, \sqrt{3} h_1 h_2^2, h_2^3]^t \]
and \[ K = [\sqrt{2} k_1, \sqrt{2} k_2,k_1^2, \sqrt{2} k_1 k_2,k_2^2,0,0,0,0]^t \, .\]

Now observe that the last four rows of $U$ can be used to define 4 linearly independent affine curves through $(0,0) \in \mathbb{C}^2$ of degree less or equal to 3. To see this, set $U = (u_{ij})$ and consider the polynomials $P_i \in \mathbb{C}[X,Y]$, $i=6,7,8,9$ defined by:
\[ P_i := u_{i 1}\sqrt{3} X + u_{i2}  \sqrt{3} Y +
u_{i 3}\sqrt{3} X^2   + u_{i4}  \sqrt{6} X Y  + u_{i5}  \sqrt{3} Y^2 +
u_{i6} X^3 + u_{i7}  \sqrt{3} X^2 Y+ u_{i8} \sqrt{3} X Y^2 + u_{i9} Y^3 \]

Then equation (\ref{curves}) implies that \[ P_i(h_1(z),h_2(z)) = 0 \]
for $i=6,7,8,9$ and $z$ near $0 \in \mathbb{C}$. Thus, $z \to (h_1(z),h_2(z))$ is a local parameterization of an irreducible connected component of each affine curve $P_i = 0$. Then the four curves share a common component $\mathcal{P}$ through $(0,0)$. So we have three possibilities $\mathrm{deg}(\mathcal{P})=1 , \mathrm{deg}(\mathcal{P})=2 $ or $\mathrm{deg}(\mathcal{P})=3 $. We will show that each of them yields a contradiction.\\

{{\bf Case} $\mathrm{deg}(\mathcal{P})=1$}. This is the same as assuming that $h_1,h_2$ are linearly dependent. So there is a function $h(z)$ and complex numbers $m_1,m_2$ such that $h_1 = m_1 h$ and $h_2 = m_2 h$. Plugging this into equation (\ref{2-3}) we get \[ (1 + |k_1(z)|^2 + |k_2(z)|^2)^2 = (1 +  |m h(z)|^2)^3 \]
where $m = \sqrt{|m_1|^2 + |m_2|^2}$. Changing the variable $z$ with  $z = h^{-1}(w)$ we get
\[ (1 + |\widetilde{k}_1(w)|^2 + |\widetilde{k}_2(w)|^2)^2 = (1 +  |m w|^2)^3 \]
a further change of variables $w = \frac{z}{m}$ gives
\[ (1 + |\widetilde{\widetilde{k}}_1(z)|^2 + |\widetilde{\widetilde{k}}_2(z)|^2)^2 = (1 +  |z|^2)^3 \]
but this is impossible. Namely, there are no holomorphic functions $\widetilde{\widetilde{k}}_1(z),\widetilde{\widetilde{k}}_2(z)$ as above since they should
give a (local) isometric and holomorphic embedding $\FS(1,1) \hookrightarrow \FS(2,\frac{3}{2})$ which contradicts Calabi's results \cite[Theorem 13, page 21]{[C]}.\\

{{\bf Case} $\mathrm{deg}(\mathcal{P})=2$}. This implies that there exist $Q \in \mathbb{C}[X,Y]$, $\mathrm{deg}(Q)=2$, such that $Q$ divides the four polynomials $P_6,P_7,P_8,P_9$. That is to say, \[ P_i = Q \cdot (a_i X + b_i Y + c_i) \]
for $i=6,7,8,9$. But then $P_6,P_7,P_8,P_9$ are not linearly independent. Contradiction.\\

{{\bf Case} $\mathrm{deg}(\mathcal{P})=3$}. In this case all polynomials are multiple of each other. Indeed, the irreducible curve parameterized by $h_1,h_2$ is an irreducible cubic hence up to multiple there is just one equation which define it. Contradiction.\\

So $\FS(2,1)$ and $\FS(2,\frac{3}{2})$ are not relatives hence conditions (\ref{111}) in Corollary \ref{necessary} are not sufficient conditions.

With the same idea we also get the following result.

\begin{theorem}\label{plane}
Let $a,b \in \mathbb{N}$, $\mathrm{gcd}(a,b)=1$, $1 < a < b$ and $a(a+3) < 4b + 2$.
Then $\FS(2,1)$ is not a relative neither of $\FS(2,\frac{b}{a})$ nor of $\FS(2,\frac{a}{b})$.
Equivalently, under the above conditions $F(2,a)$ and $F(2,b)$ are not relatives.
\end{theorem}

\end{document}